\documentclass[11pt]{amsart}
\usepackage{amscd}
\usepackage{pstricks}
\usepackage{color}
\theoremstyle{plain}
\newtheorem{thm}{Theorem}[section]
\newtheorem{theorem}[thm]{Theorem}

\newtheorem{lemma}[thm]{Lemma}

\newtheorem{proposition}[thm]{Proposition}
\theoremstyle{definition}
\newtheorem{remark}[thm]{Remark}

\newtheorem{question}[thm]{Question}

\numberwithin{equation}{section}

\newcommand{\sI}{{\mathcal I}}

\newcommand{\sL}{{\mathcal L}}

\newcommand{\sO}{{\mathcal O}}

\newcommand{\C}{{\mathbb C}}

\renewcommand{\P}{{\mathbb P}}
\newcommand{\Q}{{\mathbb Q}}
\newcommand{\R}{{\mathbb R}}

\newcommand{\Z}{{\mathbb Z}}

\title [Isomorphic quartic K3 surfaces]{Isomorphic quartic K3 surfaces in the view of Cremona and projective transformations}

\author{Keiji Oguiso}
\address{Mathematical Sciences, the University of Tokyo, Meguro Komaba 3-8-1, Tokyo, Japan and Korea Institute for Advanced Study, Hoegiro 87, Seoul, 
133-722, Korea}
\email{oguiso@ms.u-tokyo.ac.jp}

\thanks{The author is supported by JSPS Grant-in-Aid (S) No 25220701, JSPS Grant-in-Aid (S) 15H05738, JSPS Grant-in-Aid (A) 16H02141, JSPS Grant-in-Aid (B) 15H03611, and by KIAS Scholar Program.}

\begin{document}

\maketitle

\begin{abstract}
We show that there is a pair of smooth complex quartic K3 surfaces $S_1$ and $S_2$ in $\P^3$ such that $S_1$ and $S_2$ are isomorphic as abstract varieties 
but not Cremona isomorphic. We also show, in a geometrically explicit way, that there is a pair of smooth complex quartic K3 surfaces $S_1$ and $S_2$ in $\P^3$ such that $S_1$ and $S_2$ are Cremona isomorphic, but not projectively isomorphic. This work is much motivated by several e-mails from Professors Tuyen Truong and J\'anos Koll\'ar. 
\end{abstract}

\section{Introduction}

Throughout this note we work over $\C$. 

Let $X$ and $Y$ be closed subvarieties of $\P^n$, i.e., irreducible reduced closed subschemes of $\P^n$. We say that $X$ and $Y$ are {\it Cremona equivalent} (resp. {\it Cremona isomorphic}) if there is $f \in {\rm Bir}\,(\P^n)$ such that $f$ is defined at the generic point of $X$ and $f|_X : X \dasharrow Y$ is a birational map (resp. an isomorphism). We say that $X$ and $Y$ are {\it projectively equivalent} if there is $f \in {\rm Aut}\, (\P^n) = {\rm PGL}\, (n+1, \C)$ such that $f|_X : X \to Y$ is an isomorphism. 

This note is, in some sense, a continuation of our previous paper \cite{Og13}, has some overlap with an unpublished note \cite{Og12} and is much inspired by the following question asked by Tuyen Truong to me \cite{Tr16}:
\begin{question}\label{truong}
Assume that $X$ and $Y$, subvarieties of $\P^n$, are birational as abstract varieties. Are then $X$ and $Y$ Cremona equivalent in $\P^n$?
\end{question}
Answers are known in both affirmative and negative directions. This was pointed out by Massimiliano Mella to me after the first version of this note. In fact, in an affirmative direction, Mella and Polastri (\cite{MP09}) proved the following satisfactory:
\begin{theorem}\label{mella1}
Question (\ref{truong}) is affirmative if $n - \dim\, X \ge 2$.  
\end{theorem}
In a negative direction, they also proved the following (\cite{MP12}):
\begin{theorem}\label{mella2}
Let $Z$ be a smooth projective variety of dimension $n-1$. Assume that 
$2 \le n \le 15$. Then there are birational morphisms onto the images
$\varphi_i : Z \to \P^n$ ($i=1, 2$)
such that $X := \varphi_1(Z)$ and $Y := \varphi_2(Z)$ are (necessarily birational, but) not Cremona equivalent in $\P^n$.
\end{theorem}
In their construction in Theorem (\ref{mella2}), ${\rm deg}\, X \not= {\rm deg}\, Y$, and therefore, either $X$ or $Y$ has a singular point worse than canonical singularities if $K_Z$ is nef.
 
The aim of this note is to give another negative answer to Question (\ref{truong}) {\it under the stronger constraint that both $X$ and $Y$ are smooth hypersurfaces in $\P^n$ and $X$ and $Y$ are isomorphic as abstract varieties}. Note then that ${\rm deg}\, X = {\rm deg}\, Y$ if $n \ge 3$. 

Our main results are the following:

\begin{theorem}\label{main}
There are smooth quartic K3 surfaces $S_i \subset \P^3$ ($i = 1$, $2$) such that $S_1$ and $S_2$ are isomorphic as abstract varieties but they are not Cremona equivalent in $\P^3$.
\end{theorem}

\begin{theorem}\label{main3}
\par
\noindent
\begin{enumerate}
\item Let $S$ be a surface. Then, the following (a) and (b) are equivalent: 
\begin{enumerate}
\item $S$ is a smooth K3 surface with two very ample divisors $h_1$ and $h_2$ 
such that
$$((h_i, h_j)_S) = \left(\begin{array}{rr}
4 & 6\\
6 & 4
\end{array} \right)\,\, .$$
\item $S$ is isomorphic to a smooth complete intersection of four hypersurfaces $Q_k$ ($k=1$, $2$, $3$, $4$) of bidegree $(1, 1)$ of $P := \P^3 \times \P^3$:
$$S = Q_1 \cap Q_2 \cap Q_3 \cap Q_4 \subset P\,\, ,$$ 
such that the $i$-th projection $p_i|_S : S \to S_i := p_i(S) \subset \P^3$ ($i=1$, $2$) is an isomorphism and is given by the complete linear system $|h_i|$. 
\end{enumerate}
Moreover, under this equivalence, the surfaces $S_i \subset \P^3$ ($i=1$, $2$) are Cayley's K3 surfaces (\cite{Ca70}) in the sense of \cite{FGGL13}, i.e., determinantal smooth quartic surfaces.

\item For any surface $S$ in (1)(b), set $V := Q_1 \cap Q_2 \cap Q_3 
\subset P$.
Then
$$\tau := p_2|_V \circ (p_1|_V)^{-1} \in {\rm Bir}\, (\P^3) \setminus {\rm Aut}\, (\P^3)\,\, ,$$
and $S_i$ ($i = 1$, $2$) are Cremona isomorphic under $\tau$. 

\item If $Q_k$ ($k =1$, $2$, $3$, $4$) are very general hypersurfaces of bidegree $(1, 1)$ in $P = \P^3 \times \P^3$, then $S := Q_1 \cap Q_2 \cap Q_3 \cap Q_4$ satisfies the condition (1)(b) and the surfaces $S_i \subset \P^3$ ($i=1$, $2$) in (1)(b) are smooth quartic surfaces which are Cremona isomorphic but not projectively equivalent in $\P^3$. 
\end{enumerate}
\end{theorem}
\begin{remark}\label{veryspecial} 
\par
\noindent
\begin{enumerate}
\item We call a quartic surface $S \subset \P^3$ (linear) determinantal if $S = (\det M({\mathbf x}) = 0)$ for some $4 \times 4$ matrix $M({\mathbf x})$ whose entries are homogeneous linear forms of the homogeneous coordinates ${\mathbf x} = [x_1 : x_2 :x_3 : x_4]$ of $\P^3$. 
\item Theorem (\ref{main3})(1) states the result for {\it fixed} polarizations $h_1$ and $h_2$. So, as the referee pointed out, it is stronger than what we actually need for Theorem (\ref{main3})(3). I believe that Theorem (\ref{main3})(1) has its own interest and some other applications besides Theorem (\ref{main3})(3); See eg. the second paragraph after Proposition (\ref{takahashi}). Theorem (\ref{main3})(1), (2) are valid over any algebraically closed field $k$, as our proof shows. 
\item The condition {\it very general} in Theorem (\ref{main3})(3) will be made more explicit by Propositions (\ref{prop62}), (\ref{prop63}) in Section \ref{sect6}. 
\end{enumerate}
\end{remark}
Our proof of Theorem (\ref{main}) is indirect. In fact, as in \cite{Og13}, we prove Theorem (\ref{main}) by combining standard results on K3 surfaces with the following special case of a more general theorem due to Takahashi (\cite[Theorem 2.3, Remark 2.4]{Ta98}), whose proof, being based on the Noether-Fano 
inequality (\cite[Theorem 1.4]{Ta98}), is given in Appendix:
\begin{proposition}\label{takahashi}
Let $S, S' \subset \P^3$ be smooth quartic K3 surfaces 
and $\varphi \in {\rm Bir}\,(\P^3)$ such that $\varphi_*S =S'$. Here $\varphi_*S$ is the Zariski closure of the image $\varphi(\eta_S)$ of the generic point $\eta_S \in S$ in $\P^3$. Assume that $\varphi \not\in {\rm Aut}\,(\P^3)$. Then, there is an irreducible reduced curve $C \subset S$ such that ${\rm deg}\, C := (C, H)_{\P^3} < 16$ and the classes $C$ and $H|_S$ are linearly independent in ${\rm NS}\, (S)$. Here $H$ is the hyperplane class of $\P^3$. 
\end{proposition}

Theorem (\ref{main3}) is much inspired by several constructive comments from J\'anos Koll\'ar (\cite{Ko16}) on the first version of this note and beautiful works by Beauville \cite[Corollay 6.6]{Be00} on a characterization of Cayley's K3 surfaces and by D. Festi, A. Garbagnati, B. van Geemen and R. van Luijk \cite{FGGL13} on automorphisms of Cayley's K3 surfaces of Picard number $2$ (See also \cite{CT07}, \cite{BHK13} for interesting relevant works). So, the surfaces $S_i$ are known ones. The novelty of Theorems (\ref{main3}) is to provide examples of pairs of smooth quartic surfaces that are Cremona isomorphic but not projectively equivalent in geometrically simple and concrete terms. 

Quite recently, I. Shimada also informed me T. Shioda's observation that the Fermat quartic K3 surface $F$, which contains exactly 48 lines, has yet another smooth quartic surface model $F^*$ which contains exactly 56 lines (\cite{SS16}, \cite{Sh16}). It is then clear that they are isomorphic but not projectively equivalent. The corresonding two polarizations satisfy the condition in Theorem (\ref{main3}) (1). So, some isomorphism beteween $F$ and $F^*$ can be obtained as a Cremona transformation via a suitable complete interesection $\tilde{F}$ in $P$ as in Theorem (\ref{main3}) (1), (2). It may be interesting to find an explicit equations of $\tilde{F} \subset P$ and explicit determinantal descriptions of $F$ and $F^*$.  

We note that the following classical result, essentially due to Matsumura-Monsky \cite{MM63}, may justify our restriction to smooth quartic surfaces in $\P^3$ in Theorems (\ref{main}), (\ref{main3}): 
\begin{theorem}\label{mm}
Let $X$ and $Y$ be smooth hypersurfaces of $\P^n$ of degree $d$. 
Assume that $n \ge 3$ and $(n, d) \not= (3, 4)$.  Then $X$ and $Y$ are projectively equivalent, in particular, $X$ and $Y$ are Cremona isomorphic, if $X$ and $Y$ are isomorphic as abstract varieties. 
\end{theorem}

{\bf Acknowledgements.} I would like to express my thanks to Professors Tuyen Truong, Massimiliano Mella, J\'anos Koll\'ar, Xun Yu, Ichiro Shimada for valuable e-mail correspondences and their interests in this work with many constructive comments and corrections. 

\section{Notation and preliminary results.}\label{sect2} 

Throughout this note, we denote $L \otimes_{\Z} K$ by $L_K$ for a $\Z$-module $L$ and a $\Z$-algebra $K$. We denote the cyclic group of order $n$ by $\Z_n$. We call a closed point $P$ of a variety $V$ {\it general} (resp. {\it very general}) if $P$ belongs to the complement of the union of finitely many (resp. countably many) prescribed closed proper subvarieties of $V$. We denote by $\eta_V$ the generic point of the corresponding irreducible reduced scheme of $V$. 

Let $S$ be a projective K3 surface. We denote by $\sigma_S$ a non-zero holomorphic $2$-form on $S$ and by ${\rm NS}\, (S)$ the N\'eron-Severi group of $S$. The Picard group ${\rm Pic}\, (S)$ is isomorphic to ${\rm NS}\, (S)$ by the cycle map. We identify ${\rm Pic}\, (S)$ with the sublattice ${\rm NS}\, (S)$ of $H^2(S, \Z)$ with the intersection form $(*,**)_S$. The lattice $(H^2(S, \Z), (*,**)_S)$ is isomorphic to the K3 lattice $\Lambda_{{\rm K3}} := U^{\oplus 3} \oplus E_8(-1)^{\oplus 2}$. We denote the self-intersection number $(x, x)_S$ by $(x^2)_S$. The orthogonal complement 
$$T(S) := {\rm NS}\, (S)^{\perp}_{H^2(S, \Z)}$$ 
of ${\rm NS}\, (S)$ in $H^2(S, \Z)$ is the transcendental lattice. $T(S)$ is then the minimal primitive sublattice $T$ of $H^2(S, \Z)$ such that $H^0(S, \Omega_S^2) = \C \cdot \sigma_S \subset T_{\C}$. The dual lattice of ${\rm NS}\, (S)$ is 
$${\rm NS}\, (S)^* := \{x \in {\rm NS}\, (S)_{\Q}\, \vert \, (x, {\rm NS}\,(S))_S \subset \Z \}\,\, .$$ 
We have a natural inclusion ${\rm NS}\,(S) \subset {\rm NS}\,(S)^*$ and similarly for $T(S) \subset T(S)^*$. As the lattice $(H^2(S, \Z), (*,**)_S)$ is unimodular, there is a natural isomorphism 
$${\rm NS}\, (S)^*/{\rm NS}\, (S) \simeq T(S)^*/T(S)$$ 
which is compatible with the action of ${\rm Aut}\, (S)$ (\cite[Proposition 1.6.1]{Ni79}). 

The positive cone $P(S)$ is the connected component of the subset 
$\{x \in {\rm NS}\, (S)_{\R}\, \vert\, (x^2)_S > 0\}$ 
of ${\rm NS}\, (S)_{\R}$, containing the ample classes. Let $\overline{P}(S)$ 
be the closure of $P(S)$ in the topological vector space ${\rm NS}\, (S)_{\R}$. The nef cone $\overline{{\rm Amp}}\,(S)$ is the closure of the ample cone ${\rm Amp}\,(S)$ in ${\rm NS}\, (S)_{\R}$. Note that $\overline{{\rm Amp}}\,(S) \subset \overline{P}(S)$.

The following lemma is well-known and is proved in several ways (see eg. \cite[Proposition 2.4]{Og14}). 
\begin{lemma}\label{lem43} 
Let $S$ be a projective K3 surface and $g \in {\rm Aut}\, (S)$ such that $g^* \vert_{{\rm NS}\, (S)}$ is of finite order. Then $g$ is of finite order.
\end{lemma}
Our references on basic facts on K3 surfaces, their projective models and lattice polarized K3 surfaces are \cite[Chapter VIII]{BHPV04}, \cite{SD74} and \cite[Sections 1-3]{Do96} respectively. 

\section{Proof of Theorem (\ref{main}).}\label{sect3}

In this section we shall prove Theorem (\ref{main}). Let $\ell$ be an integer such that $\ell \ge 5$. Choose and fix such an $\ell$. We freely use the notation introduced in Section \ref{sect2}.

Throughout this section, $S$ is a K3 surface such that ${\rm NS}\, (S) = L$,  
where 
$$ L := \Z h_1 + \Z h_2\,\, ,\,\, 
((h_i, h_j)_S) = \left(\begin{array}{rr}
4 & 4\ell\\
4\ell & 4
\end{array} \right)\,\, .$$ 
As the lattice $L$ is even and of signature 
$(1,1)$, there is a unique primitive embedding $L \to \Lambda_{{\rm K3}}$, up to ${\rm O}(\Lambda_{{\rm K3}})$ and such K3 surfaces $S$ certainly exist (\cite[Corollary 2.9]{Mo84}, \cite[Theorem 3.1]{Ni79}). They are all projective and very general in the $18$-dimensional family of the $L$-polarized K3 surfaces (\cite[Sections 1-3]{Do96}). 

{\it In what follows, we choose and fix a very general $S$ which also enjoys the property in}:

\begin{lemma}\label{lem41}
$g^* \sigma_S = \pm \sigma_S$ and $g^* | T(S) = \pm id$ for all $g \in {\rm Aut}\, (S)$, if $S$ is very general. 
\end{lemma}

\begin{proof} By the minimality of $T(S)$, it suffices to show that $g^* \sigma_S = \pm \sigma_S$ if $S$ is very general. 

We have $g^*\sigma_S = \alpha \sigma_S$ for some $\alpha \in \C$. As $S$ is projective, $\alpha$ is a cyclotomic integer (\cite[Theorem14.10]{Ue75}). We have
$\sigma_S \in V(\alpha) \subset T(S) \otimes_{\Z} \C$.  
Here $V(\alpha)$ is the eigenspace of $g^*|T(S)$ of eigenvalue $\alpha$. 
The space $V(\alpha)$ is a proper linear subspace of $T(S) \otimes_{\Z} \C$ if $\alpha \not= \pm 1$, as $\alpha$ has then a Galois conjugate $\beta$ with $\beta \not= \alpha$. Thus, the set of periods of all such $S$ that $\alpha \not= \pm 1$ belongs to some countable union of hypersurfaces in the $18$-dimensional period domain of $L$-polarized K3 surfaces. This implies the result. 
\end{proof} 
Replacing $(h_1, h_2)$ by $(-h_1, -h_2)$ if necessary, we may and will assume that $h_1 \in P(S)$. By the shape of ${\rm NS}\, (S)$, one readily obtains the following:

\begin{lemma}\label{lem44}
\par
\noindent
\begin{enumerate}
\item Let $xh_1 + yh_2 \in {\rm NS}\, (S)_{\R}$. Then 
$$((xh_1 + yh_2)^2)_S = 4x^2 + 8\ell xy + 4y^2 
= 4((x+\ell y)^2 - (\ell^2 -1)y^2)\,\, .$$
\item 
${\rm NS}\, (S)$ represents neither $0$ nor $\pm 2$, i.e., $(d^2)_S \not= 0, \pm 2$ for $d \in {\rm NS}\, (S) \setminus \{0\}$. 
\item
$$T(S)^*/T(S) \simeq {\rm NS}\,(S)^*/{\rm NS}\, (S) = \langle \frac{h_1}{4}, \frac{h_2 - \ell h_1}{4(\ell^2 -1)} \rangle \simeq \Z_{4} \oplus \Z_{4(\ell^2 -1)}
\,\, .$$
\end{enumerate} 
\end{lemma}

We set $v_1 := (-\ell + \sqrt{\ell^2 -1})h_1 + h_2$, $v_2 := h_1 + (-\ell + \sqrt{\ell^2 -1})h_2 \in {\rm NS}\,(S)_{\R}$.

\begin{lemma}\label{lem45} $\overline{{\rm Amp}}\, (S) = {\overline P}\, (S) = \R_{\ge 0}v_1 + \R_{\ge 0}v_2$. 
\end{lemma}

\begin{proof}
As $h_1 \in P(S)$, the second equality follows from Lemma (\ref{lem44})(1). As ${\rm NS}\, (S)$ does not represent $-2$ by Lemma (\ref{lem44})(2), $S$ contains no $\P^1$. This implies the first equality. 
\end{proof}

\begin{lemma}\label{lem46}
${\rm Aut}\, (S)$ has no element of finite order other than $id_S$.
\end{lemma}
\begin{proof}
Let $g \in {\rm Aut}\, (S)$. Then, either $g^*v_1 = \alpha v_1$ and $g^*v_2 = \beta v_2$ (first case) or $g^*v_1 = \alpha v_2$ and $g^*v_2 = \beta v_1$ (second case), for some positive real numbers $\alpha$ and $\beta$. 

{\it Assume that $g$ is of finite order.} 

Then, in the first case, $\alpha = \beta = 1$, whence $g^* \vert {\rm NS}\, (S) = id$. Then $g^* \vert_{{\rm NS}\, (S)^*/{\rm NS}\, (S)} = id$. Hence $g^* \vert_{T(S)^*/T(S)} = id$ as well. On the other hand, $g^* \vert_{T(S)} = \pm id$, by Lemma (\ref{lem41}). It follows that $g^* \vert_{T(S)} = id$ by Lemma (\ref{lem44})(3), as $g^*|_{T(S)^*/T(S)} = id$. Hence $g = id_S$ by the global Torelli theorem for K3 surfaces. 

In the second case, $(g^2)^*v_1 = \alpha^2 v_1$ and $(g^2)^*v_2 = \beta^2 v_2$. 
Hence $g^2 = id_S$ as we have shown. Thus $\alpha = \beta = 1$. Therefore 
$g^*v_1 = v_2$ and $g^*v_2 = v_1$.  
This implies that
$$g^*h_1 = h_2\,\, ,\,\, g^*h_2 = h_1\,\, .$$
Then $g^*|_{{\rm NS}\,(S)^*/{\rm NS}\,(S)} \not= \pm id$ by Lemma (\ref{lem44})(3). Indeed, we have $(h_1 \pm h_2)/4 \not\in {\rm NS}\, (S)$, as $h_1$ and 
$h_2$ form $\Z$-basis of ${\rm NS}\, (S)$. On the other hand, $g^*|_{T(S)^*/T(S)} = \pm id$, as $g^*|_{T(S)} = \pm id$ by our choice of $S$, a contradiction. This proves the assertion. 
\end{proof}
We set ${\mathcal H} := \{h \in \overline{{\rm Amp}}\, (S) \cap 
{\rm NS}\, (S)\, 
\vert\, (h^2)_S = 4\}$. Then, $h_1, h_2 \in {\mathcal H}$ by Lemma (\ref{lem45}) and the action of ${\rm Aut}\,(S)$ on ${\rm NS}\, (S)_{\R}$ preserves ${\mathcal H}$. 
\begin{lemma}\label{lem47}
\par
\noindent
\begin{enumerate}
\item
Any $h \in {\mathcal H}$ is very ample. 
\item There is no $g \in {\rm Aut}\, (S)$ such that $g^*h_1 = h_2$. 
\end{enumerate}
\end{lemma}

\begin{proof}
As $S$ contains no $\P^1$, the complete linear system $\vert h \vert$ is free (\cite[2.7]{SD74}). Also, there is no $d \in {\rm NS}\, (S) \setminus \{0\}$ 
with 
$(d^2)_S \in \{0, \pm 2 \}$ by Lemma (\ref{lem44})(2). The assertion (1) then follows from \cite[Theorem 5.2]{SD74}. 

Let us show the assertion (2). {\it Assume to the contrary that} there is $g \in {\rm Aut}\, (S)$ such that $g^*h_1 = h_2$. Write $g^*h_2 = ah_1 + bh_2$. Then
$$g^*(h_1, h_2) = (h_1, h_2)\left(\begin{array}{rr}
0 & a\\
1 & b
\end{array} \right)\,\, .$$
As ${\rm det}\, g^*|_{{\rm NS}\, (S)} = \pm 1$, it follows that $a = \pm 1$. 
We have $b = -\ell(a-1)$ by
$$4\ell = (h_1,h_2)_S = (g^*h_1,g^*h_2)_S = (h_2, ah_1+bh_2)_S = 4a\ell + 4b\,\, .$$
Hence $(a, b)$ is either $(1, 0)$ or $(-1, 2\ell)$. 

Assume that $(a, b) = (1, 0)$. Then $(g^*)^2h_1 = h_1$ and $(g^*)^2h_2 = h_2$. Thus $g^2 \vert_{{\rm NS}\, (S)} = id$. Hence $g$ is of finite order by Lemma (\ref{lem43}). However, then $g = id_S$ by Lemma (\ref{lem46}), a contradiction to $h_1 \not= h_2$ in ${\rm NS}\, (S)$. 

Assume that $(a, b) = (-1, 2\ell)$. Then, by $g^*h_1 = h_2$ and $g^*h_2 = -h_1 + 2\ell h_2$, we have 
$$(g^*)^2\frac{h_1}{4} = g^*\frac{h_2}{4} = \frac{-h_1 + 2\ell h_2}{4}\,\, .$$
By Lemma (\ref{lem41}), we have $(g^{*})^2 = id$ on $T(S)^*/T(S) \simeq {\rm NS}\, (S)^*/{\rm NS}\, (S)$. Hence
$$\frac{-h_1 + \ell h_2}{2} = (g^{*})^2\frac{h_2}{4} - \frac{h_1}{4} \in {\rm NS}\, (S)\,\, ,$$
a contradiction to the fact that $h_1$ and $h_2$ form $\Z$-basis of ${\rm NS}\, (S)$. This proves (2).
\end{proof}
Let $h \in {\mathcal H}$ and $\Phi_{\vert h \vert} : S \to \P^3$ be the embedding defined by the complete linear system $\vert h \vert$ (Lemma (\ref{lem47})(1)). We set $S_h := \Phi_{\vert h \vert}(S) \subset \P^3$. Then $S_h$ is a smooth quartic K3 surface isomorphic to $S$, as $h$ is very ample and $(h^2)_S = 4$. 
\begin{lemma}\label{lem48}
Let $H$ be the hyperplane class of $\P^3$ and $C$ be any effective curve on $S_h$ such that $(C, H)_{\P^3} < 16$. Then, the classes $C$ and $H|_S$ are linearly dependent in ${\rm NS}\, (S_h)$. 
\end{lemma}

\begin{proof} Identify $S_h$ with $S$ by $\Phi_{\vert h \vert}$. {\it Assuming to the contrary that there would be an effective curve $C \subset S_h$ such that the class $c := [C]$ is linearly independent to 
$h = H \vert_{S}$ in ${\rm NS}\, (S)$, we shall derive a contradiction.} Then $N := \langle c, h \rangle$ is a sublattice of ${\rm NS}\, (S)$ of the same rank $2$. In particular, the signature is $(1,1)$.
Then  
$$\vert N \vert := \vert {\rm det}\left(\begin{array}{rr}
(c^2)_S & (c,h)_S\\
(c,h)_S & (h^2)_S
\end{array} \right) \vert = (c,h)_S^2 - (c^2)_S \cdot (h^2)_S > 0$$
and $\vert N \vert$ is divisible by $\vert {\rm NS}(S) \vert := \vert {\rm det}(h_i.h_j) \vert$. 
Since $c = [C]$ is an effective class and $C \not\simeq \P^1$ 
by Lemma (\ref{lem44})(2), we have $(c^2)_S \ge 0$. Hence 
$$\vert N \vert \le (c,h)_S^2 = (C,H)_{\P^3}^2 < 16^2\,\, .$$
On the other hand, by $\ell \ge 5$, 
$$\vert {\rm NS}(S) \vert = 16\ell^2 - 16 > 
16 \cdot 17 - 16 = 16^2\,\, .$$
However, then $\vert N \vert$ is not divisible by 
$\vert {\rm NS}(S) \vert$, 
a contradiction. 
\end{proof}  
Let $S_i := \Phi_{\vert h_i \vert}(S) \subset \P^3$ ($i=1$, $2$). Then $S_1$ and $S_2$ are smooth quartic K3 surfaces in $\P^3$ such that $S_1 \simeq S_2 \simeq S$. The next lemma completes the proof of Theorem (\ref{main}):

\begin{lemma}\label{lem49}
There is no $\varphi \in {\rm Bir}\, (\P^3)$ such that $\varphi_*S_2 = S_1$. 
\end{lemma}

\begin{proof} {\it Assuming to the contrary that there would be 
$\varphi \in {\rm Bir}\, (\P^3)$ such that $\varphi_*(S_2) = S_1$, we shall derive a contradiction.} 

By Proposition (\ref{takahashi}) and Lemma (\ref{lem48}), $\varphi \in {\rm Aut}\, (\P^3)$. Hence
$$g :=  (\Phi_{\vert h_1 \vert}|_{S_1})^{-1} \circ (\varphi |_{S_2}) \circ \Phi_{\vert h_2 \vert} \in {\rm Aut}\, (S)\,\, .$$
We also note that $h_1 = \Phi_{\vert h_1 \vert}^*H$, $h_2 = \Phi_{\vert h_2 \vert}^*H$ for the hyperplane class $H$ of $\P^3$. We have $\varphi^*H = H$, as $\varphi \in {\rm Aut}\, (\P^3) = {\rm PGL}(4, \C)$. However, then $g^*h_1 = h_2$, a contradiction to Lemma (\ref{lem47})(2). 
\end{proof}

\section{Proof of Theorem (\ref{main3})(1).}\label{sect4}

In this section we shall prove Theorem (\ref{main3})(1).
We freely use the notation introduced in Section \ref{sect2} and Theorem (\ref{main3}). 
Let 
$$P = \P^3 \times \P^3 = \P_1^3 \times \P_2^3\,\, ,$$
and $p_i : P \to \P_i^3$ be the $i$-th projection ($i=1$, $2$). 
We denote the homogeneous coordinates of $P$ by
$$({\mathbf x}, {\mathbf y}) = ([x_1 :  x_2 : x_3 : x_4], [y_1 :  y_2 : y_3 : y_4])\,\, .$$
We denote by ${\mathbf x}^t$ (resp. ${\mathbf y}^t$) the transpose of ${\mathbf x}$ (resp. ${\mathbf y}$). 

Let $H_1$ and $H_2$ be divisors on $P = \P^3 \times \P^3$ of bidegree $(1, 0)$ and $(0, 1)$ respectively. Then 
$$\sO_{P}(H_1) = \sO_{P}(1, 0) = p_1^*\sO_{\P_1^3}(1)\,\, ,\,\, \sO_{P}(H_2) = \sO_{P}(0, 1) = p_2^*\sO_{\P_2^3}(1)\,\, .$$

{\it First, we show that (b) implies (a).} 

$S$ is a K3 surface by the adjunction formula. Let $\ell_i$ be the hyperplane class of $S_i \subset \P^3$ and $h_i = p_i^*(\ell_i)$. Then $h_1$ and $h_2$ satisfy the requirement in (a).

{\it Next we show that (a) implies (b).}

Consider the embedding
$$\Phi := \Phi_{|h_1|} \times \Phi_{|h_2|} : S \to P = \P^3 \times \P^3$$
given by the very ample complete linear systems $|h_i|$.

 We denote $\tilde{S} := \Phi(S)$. Then we have $h_i = H_i |_{\tilde{S}}$ ($i=1$, $2$) and 
$$S_i := p_i(\tilde{S}) = \Phi_{|h_i|}(S) \simeq S\,\, .$$
{\it In what follows, whenever we write $S = \tilde{S}$ (resp. $S = S_i$ ($i=1$, $2$)), we understand that the identification is made by $\Phi$ (resp. $\Phi_{|h_i|}$).} 

We note that 
$$\C^4 \simeq \C \langle x_1, x_2, x_3, x_4 \rangle = H^0(\P_1^3, \sO_{\P_1^3}(1)) \simeq H^0(S, \sO_S(h_1))$$ under the restriction map and the identification of $S$ with $S_1$ by $\Phi_{|h_1|}$. Similarly
$$\C^4 \simeq \C \langle y_1, y_2, y_3, y_4 \rangle = H^0(\P_2^3, \sO_{\P_2^3}(1)) \simeq H^0(S, \sO_S(h_2))\,\, .$$
Now we identify $S$ with $S_1$ by $\Phi_{|h_1|}$. 
As $h_1$ and $h_2$ are very ample with $(h_i^2)_S = 4$ and $(h_1, h_2)_S = 6$, we have an exact sequence of sheaves on $\P_1^3$:
$$0 \to \sO_{\P_1^3}(-1) \otimes_{\C} H^0(\P_2^3, \sO_{\P_2^3}(1)) \stackrel{M({\mathbf x}) \cdot}{\to} \sO_{\P_1^3} \otimes_{\C} H^0(\P_2^3, \sO_{\P_2^3}(1)) \stackrel{m}{\to} \sO_S(h_2) \to 0\,\, ,$$
by Beauville \cite[Corollary 6.6, Proposition 1.11]{Be00}. Then, by tensoring $\sO_{\P_1^3}(1)$, we obtain an exact sequence:
$$0 \to \sO_{\P_1^3} \otimes_{\C} H^0(\P_2^3, \sO_{\P_2^3}(1)) \stackrel{M({\mathbf x}) \cdot}{\to} \sO_{\P_1^3}(1) \otimes_{\C} H^0(\P_2^3, \sO_{\P_2^3}(1)) \stackrel{m}{\to} \sO_S(h_1+h_2) \to 0\,\, .$$
In both exact sequences, the first map $M({\mathbf x}) \cdot$ is the multiplication map 
$${\mathbf y}^t \mapsto M({\mathbf x}) \cdot {\mathbf y}^t$$
by a $4 \times 4$ matrix $M({\mathbf x}) = (m_{ij}({\mathbf x}))$ whose $(i,j)$-entry 
$$m_{ij}({\mathbf x}) = a_{ij1}x_1 + a_{ij2}x_2 + a_{ij3}x_3 + a_{ij4}x_4$$ 
is a linear homogeneous form of ${\mathbf x}$. The second map $m$ is the natural multiplication map through the isomorphism $H^0(\P_2^3, \sO_{\P_2^3}(1)) \simeq H^0(S, \sO_{S}(h_2))$ induced by the restriction map and $\Phi_{|h_2|}^*$. 
Taking the cohomology exact sequence of the second exact sequence above, we obtain the exact sequence
$$0 \to \C \otimes_{\C} H^0(\P_2^3, \sO_{\P_2^3}(1)) \stackrel{M({\mathbf x}) \cdot}{\to} H^0(P, \sO_P(1,1)) \stackrel{m}{\to} H^0(S, \sO_S(h_1+h_2)) \to 0\,\, .$$
Here we used $H^1(\P^3, \sO_{\P^3}) = 0$, and the Kunneth isomorphism for the middle factor above:
$$H^0(\P_1^3, \sO_{\P_1^3}(1)) \otimes_{\C} H^0(\P_2^3, \sO_{\P_2^3}(1)) \simeq  H^0(P, \sO_P(1,1))\,\, .$$
Thus 
$$\tilde{S} \subset (M({\mathbf x})\cdot {\mathbf y}^t = {\mathbf 0}^t) 
= Q_1 \cap Q_2 \cap Q_3 \cap Q_4 \subset P\,\, .$$
Here $Q_k$ is a hypersurface of bidegree $(1,1)$ defined by 
$$Q_k = ({\mathbf x}\cdot (a_{ijk})_{i,j} \cdot {\mathbf y}^t = 0) \subset P\,\, .$$   
Note that  
$$\sum_{j=1}^{4}(\sum_{k=1}^{4}a_{ijk}x_k)y_j = \sum_{k=1}^{4}(\sum_{j=1}^{4}a_{ijk}y_j)x_k\,\, .$$
Then we have an identity
$$M({\mathbf x})\cdot {\mathbf y}^t = N({\mathbf y})\cdot {\mathbf x}^t\,\, .$$
Here $N({\mathbf y}) =(n_{ik}({\mathbf y}))_{i, k}$ is the $4 \times 4$ matrix whose $(i, k)$-entry is
$$n_{ik}({\mathbf y}) = a_{i1k}y_1 + a_{i2k}y_2 + a_{i3k}y_3 + a_{i4k}y_4\,\, .$$ 
The next proposition completes the proof of the fact that (a) implies (b) and the last statement of Theorem (\ref{main3})(1). Our trivial identity above plays an important role in the proof and also shows a way to obtain explicit equations of $S_i$ ($i=1$, $2$) from $\tilde{S}$. 
\begin{proposition}\label{lem32}
As closed subschemes, we have:
\begin{enumerate}
\item $S_1 = (\det M({\mathbf x}) = 0)$ in $\P_1^3$. 
\item $\tilde{S} = Q_1 \cap Q_2 \cap Q_3 \cap Q_4$ in $P$. 
\item $S_2 = (\det N({\mathbf y}) = 0)$ in $\P_2^3$.
\end{enumerate} 
\end{proposition} 
\begin{proof}
As $\sO_{S_1}(h_1+h_2)$ is torsion as a sheaf on $\P^3$, the matrix $M({\mathbf a})$ is of rank $4$ for general ${\mathbf a} \in \P^3$. Thus $(\det M({\mathbf x}) = 0)$ is a hypersurface of degree $4$ in $\P^3$.  Set 
$$T := (\det M({\mathbf x}) = 0) \subset \P^3\,\, .$$ 
Let ${\mathbf a} \in S_1$. Then there is a point ${\mathbf b} \in \P^3$ such that $({\mathbf a}, {\mathbf b}) \in \tilde{S}$, as $S_1 = p_1(\tilde{S})$. As $\tilde{S} \subset (M({\mathbf x}) \cdot {\mathbf y}^{t} = 0)$, it follows that 
$$M({\mathbf a}) \cdot {\mathbf b}^{t} = 0\,\, .$$
As ${\mathbf b} \not= (0,0,0,0)$ as vectors, it follows that $\det M({\mathbf a}) = 0$. Hence $S_1 \subset T$ as sets. As both $S_1$ and $T$ are hypersurfaces of degree $4$, it follows that $S_1 = T$ as schemes. This shows (1). 

We show (2). As $S_1 = (\det M({\mathbf x}) = 0)$ is smooth, $M({\mathbf a})$ is of rank $3$ for each ${\mathbf a} \in S_1$. This directly follows from the Jacobian criterion and the chain rule applied for the cofactor expansion of $\det M({\mathbf x})$ (See eg. the first four lines in the proof of \cite[Proposition 2.2]{FGGL13}).
Set, as closed subschemes,
$$W := Q_1 \cap Q_2 \cap Q_3 \cap Q_4 \subset P\,\, .$$ 
Then $p_1(W) = S_1$, as the defining equation of $W$ is $M({\mathbf x})\cdot {\mathbf y}^t = 0$ and ${\mathbf y} \not= 0$ as vectors. 
Choose and fix a point ${\mathbf a} \in S_1$. Again, as $M({\mathbf a})$ is of rank $3$, there is exactly one point ${\mathbf b} \in \P^3$ such that $M({\mathbf a}) \cdot {\mathbf b}^t = 0$, i.e., the fiber $(p_1|_W)^{-1}({\mathbf a})$ is exactly one point. As $p_1 |_{\tilde{S}} : \tilde{S} \to S_1$ is an isomorphism and $\tilde{S} \subset W$, it follows that $\tilde{S} = W_{{\rm red}}$. Here $W_{{\rm red}}$ is the reduction of $W$. In particular, $W$ is of pure dimension $2$. It follows that $W = Q_1 \cap Q_2 \cap Q_3 \cap Q_4$ is a complete intersection as schemes. 

We now show that ${\tilde S} = W$ as schemes. 
We have a natural exact sequence
$$0 \to \sI \to \sO_W \to \sO_{\tilde{S}} \to 0\,\, ,$$
as ${\tilde S}$ is a closed subscheme of $W$. 
Tensoring the invertible sheaf $\sO_{P}(n, n)$ ($n \in \Z_{>0}$), we have an exact sequence
$$0 \to \sI(n,n) \to \sO_W(n, n) \to \sO_{\tilde{S}}(n, n) \to 0\,\, .$$
Note that $\sO_{\tilde{S}}(n, n) = \sO_S(n(h_1+h_2))$ under the identification $\tilde{S} = S$ by $\Phi$. Then, we have  
$$h^0(\tilde{S}, \sO_{\tilde S}(n, n)) = h^0(S, \sO_S(n(h_1+h_2))) = 
\frac{n^2((h_1+h_2)^2)_S}{2} + \chi(\sO_S) = 10n^2 +2\,\, .$$
As $W$ is a complete intersection of four hypersurfaces $Q_i$ of bidegree $(1, 1)$ in $P$, we also readily obtain that
$$h^0(W, \sO_{W}(n, n)) = \chi(\sO_{W}(n, n))) = 10n^2 +2$$ 
for all sufficiently large $n$, say $n \ge n_1$.  Hence
$$h^0(W, \sO_{W}(n, n)) = h^0(\tilde{S}, \sO_{\tilde{S}}(n, n))\,\, ,$$
for all $n \ge n_1$. As $\sO_P(1,1)$ is ample, there is an integer $n_2 \ge n_1$ such that $H^1(W, \sI(n, n)) = 0$ for all $n \ge n_2$. Then $H^0(W, \sI(n, n)) = 0$ for all $n \ge n_2$ from 
$$0 \to H^0(W, \sI(n, n)) \to H^0(W, \sO_W(n, n)) \to H^0(\tilde{S}, \sO_{\tilde{S}}(n, n)) \to H^1(W, \sI(n, n)) = 0\,\, .$$
It follows that $\sI = 0$ as sheaves, as $\sO_{P}(1,1)$ is ample. 
Hence $\tilde{S} = W$ as claimed. This completes the proof of (2).  

We show (3). We have 
$\tilde{S} = (M({\mathbf x}) \cdot {\mathbf y}^t = {\mathbf 0}^t) = (N({\mathbf y}) \cdot {\mathbf x}^t = {\mathbf 0}^t)$
by (2) and by definition of $N({\mathbf y})$. Then 
$S_2 = p_2(\tilde{S}) \subset (\det N({\mathbf y}) = 0) \subset \P_2^3$. The matrix $N({\mathbf b})$ is of rank $4$ for each ${\mathbf b} \in \P_2^3 \setminus S_2$. Indeed, otherwise, for some ${\mathbf b} \in \P_2^3 \setminus S_2$, there is a vector ${\mathbf a} \not= {\mathbf 0}$ such that $N({\mathbf b}) \cdot {\mathbf a}^t = {\mathbf 0}^t$. This means that there is a point $({\mathbf a}, {\mathbf b}) \in \tilde{S}$ such that ${\mathbf b} \not\in S_2$, a contradiction to the fact that $S_2 = p_2(\tilde{S})$. 

Hence $\det N({\mathbf y})$ is a non-zero homogeneous polynomial of degree $4$. As $S_2$ is a smooth quartic surface, it follows that $S_2 = (\det N({\mathbf y}) = 0)$ as claimed. 
\end{proof}

\section{Proof of Theorem (\ref{main3})(2).}\label{sect5}
In this section we shall prove Theorem (\ref{main3})(2).
We freely use the notation introduced in Section \ref{sect2} and Theorem (\ref{main3}). 

By Theorem (\ref{main3})(1), $S_i \subset \P^3$ are determinantal smooth quartic surfaces in $\P^3$. 

$V = Q_1 \cap Q_2 \cap Q_3 \subset P$ is a complete intersection as so is $S = V \cap Q_4$ and $Q_4$ is ample. As $S = V \cap Q_4$ is smooth, it follows that $V$ is normal. Indeed, otherwise, $V$ would have a singular locus of codimension one, as $V$ is a complete intersection, hence Cohen-Macaulay. Then, the ample Cartier divisor $Q_4$ intersects with the singular locus of $V$, at which $V \cap Q_4$ is necessarily singular, as $Q_4$ is Cartier. This contradicts the smoothness of $S$. Hence $V$ is normal.

The morphism $p_1|_V : V \to \P^3$ is birational, as the matrix $M({\mathbf a})$ is of rank $4$ for general ${\mathbf a} \in \P^3$ (See Section \ref{sect4} for the definition of $M({\mathbf x})$.) Changing the roles of ${\mathbf x}$ and ${\mathbf y}$ by taking the transpose, we see that the morphism $p_2|_V : V \to \P^3$ is also birational. Thus $\tau \in {\rm Bir} (\P^3)$. 

Let $H$ be the hyperplane bundle of $\P^3$ and $H_i := p_i^*H$. Then $H_i|_V$ are line bundles on $V$ with intersection numbers
$$((H_1|_V)^3)_V = (H_1^3(H_1+H_2)^3)_{P} = 3\,\, ,\,\, ((H_1|_V)^2H_2|_V)_V = 
(H_1^2H_2(H_1+H_2)^3)_{P} = 6\,\, .$$
If $\tau \in {\rm Aut} (\P^3)$, then $\tau^*H = H$. As $\tau \circ (p_1|_V) = p_2|_V$, it follows that $H_1|V = H_2|V$ as line bundles on $V$. However, then 
$((H_1|_V)^3)_V = ((H_1|_V)^2H_2|V)_V$, a contradiction to the computation above. Hence $\tau \not\in {\rm Aut} (\P^3)$.

As $V$ is normal, the fibers of $p_1|_V$ are connected by the Zariski main theorem. Thus $(p_1|_V)^{-1}({\mathbf a})$ is connected for all ${\mathbf a} \in \P^3$. As $S_1$ is of codimension one in $\P^3$ and $p_1|_V$ is birational, it follows that $(p_1|_V)^{-1}$ is an isomorphism over some Zariski open dense subset of $S_1$, in particular, at the generic point $\eta_{S_1}$ of $S_1$. As $p_1|_V(\eta_S) = \eta_{S_1}$ by $(p_1|_V)(S) = S_1$, it follows that $(p_1|_V)^{-1}(\eta_{S_1}) = \eta_S$, and therefore, $\tau(\eta_{S_1}) = \eta_{S_2}$ by $p_2|_V(\eta_S) = \eta_{S_2}$. Thus $\tau$ induces a birational map $\tau|_{S_1}$ from $S_1$ to $S_2$. The map $\tau|_{S_1}$ is then an isomorphism, as both $S_i$ are smooth surfaces with nef canonical divisors. 
This completes the proof of Theorem (\ref{main3})(2).

\section{Proof of Theorem (\ref{main3})(3).}\label{sect6}

In this section we shall prove Theorem (\ref{main3})(3). We freely use the notation introduced in Section \ref{sect2} and Theorem (\ref{main3}). 

Throughout this section, $S = Q_1 \cap Q_2 \cap Q_3 \cap Q_4$ is a complete intersection 
in $P =\P^3 \times \P^3$ of four hypersurfaces $Q_k$ of bidegree $(1,1)$. We denote by $H$ the hyperplane bundle of $\P^3$. We set $H_i = p_i^*H$ and 
$h_i = H_i|_S$. Here $p_i : P \to \P^3$ is the $i$-th projection. Then 
$$((h_i,h_j)_S) = \left(\begin{array}{rr}
4 & 6\\
6 & 4
\end{array} \right)\,\, ,$$
provided that $S = Q_1 \cap Q_2 \cap Q_3 \cap Q_4$ is a complete intersection 
in $P$.  

Let $V = Q_1 \cap Q_2 \cap Q_3 \subset P$. If $Q_k \subset P$ ($k=1$, $2$, $3$, $4$) are general ({\it here we do not need they are very general}), $V$ is a smooth Fano threefold and $S$ is a smooth K3 surface, by the Bertini theorem and the adjunction formula. 
\begin{proposition}\label{prop62}
If $Q_k \subset P$ ($k=1$, $2$, $3$) are general and $Q_4$ is very general, 
then 
${\rm NS}\, (S) = \Z h_1 \oplus \Z h_2$.  
\end{proposition}
\begin{proof} As ${\rm Pic}\, (P) = \Z H_1 \oplus \Z H_2$, we have 
${\rm Pic}\, (V) = \Z H_1|_V \oplus \Z H_2|_V$ 
by the Lefschetz hyperplane theorem. Observe that the natural restriction map 
$H^0(\sO_P(H_1+H_2)) \to H^0(\sO_V((H_1+H_2)|_V))$
is surjective. Then $S = Q_4|_V$ is also very general in the very ample linear system $|(H_1+H_2)|_V|$. In addition, we have $H^{2,0}(S) \not= 0$ but $H^{2,0}(V) = 0$, as $S$ is a K3 surface and $V$ is a Fano threefold. Now we can apply the Noether-Lefschetz theorem (\cite[Theorem 3.33]{Vo03}) for $S \subset V$ and obtain the result. 
\end{proof}
Set $S_i = p_i(S) \subset \P^3$ ($i=1$, $2$). Note that if $S$ satisfies the condition (1)(b) in Theorem (\ref{main3}), then $S_i$ are smooth quartic K3 surfaces which are Cremona isomorphic by Theorem ({\ref{main3})(2). So, Theorem ({\ref{main3})(3) now follows from:
\begin{proposition}\label{prop63}
If $S$ is smooth and ${\rm NS}\, (S) = \Z h_1 \oplus \Z h_2$, 
then $S$ satisfies the condition (1)(b) in Theorem (\ref{main3}) and $S_1$ and $S_2$ are not projectively equivalent in $\P^3$.  
\end{proposition}
\begin{proof}
As $h_i = p_i^*H|_S$, the complete linear systems $\vert h_i \vert$ are free and $p_i = \Phi_{|h_i|}$. There is no $d \in {\rm NS}\, (S) \setminus \{0\}$ with 
$(d^2)_S \in \{0, \pm 2 \}$, by ${\rm NS}\, (S) = \Z h_1 \oplus \Z h_2$ and by the shape of $((h_i,h_j)_S)$. Thus $h_i$ are very ample by \cite[Theorem 5.2]{SD74}. Hence $S$ satisfies the condition (1)(b) in Theorem (\ref{main3}).

We show that $S_i$ ($i=1$, $2$) are not projectively equivalent. Using the shape of $((h_i,h_j)_S)$ again, we also find that 
$${\rm NS}\, (S)^*/{\rm NS}\, (S) = \langle \frac{h_1}{2} \rangle \oplus \langle \frac{h_2}{2} \rangle \oplus \langle \frac{h_1 + h_2}{5} \rangle \simeq \Z_2 \oplus \Z_2 \oplus \Z_5\,\, .$$
Moreover, by \cite[Theorem 1.1 and proof]{FGGL13}, we have
 
\begin{lemma}\label{lem21}
$f^*\sigma_S = \pm \sigma_S$, whence $f^{*} |_{{\rm NS}\, (S)^*/{\rm NS}\, (S)}) = \pm id_{{\rm NS}\, (S)^*/{\rm NS}\, (S)}$ for any $f \in {\rm Aut}\, (S)$.
\end{lemma}
\begin{lemma}\label{lem22}
There is no $f \in {\rm Aut}\, (S)$ such that $f^*h_1 = h_2$.  
\end{lemma}
\begin{proof} 
Otherwise, $f^*|_{{\rm NS}(S)^*/{\rm NS}\, (S)}$ would satisfy 
$$f^*\frac{h_1}{2} = \frac{h_2}{2}\,\, ,$$
in particular, $f^*|_{{\rm NS}(S)^*/{\rm NS}\, (S)} \not= \pm id$, a contradiction to Lemma \ref{lem21}. 
\end{proof}
If $g(S_2) = S_1$ under $g \in {\rm Aut}\, (\P^3) = {\rm PGL}(4, \C)$, then $g^*H= H$. However, then $f^*h_1 = h_2$ for $f = (p_1|_S)^{-1} \circ g \circ (p_2|_S) \in {\rm Aut}\, (S)$, a contradiction to Lemma (\ref{lem22}). This completes the proof of Proposition (\ref{prop63}). 
\end{proof}

\section{Appendix - Proof of Proposition (\ref{takahashi}).}\label{sect7}

In this appendix, we shall give a proof of Proposition (\ref{takahashi}) as it is not explicit in \cite{Ta98}. However, we should emphasize that all the arguments below are found in \cite{Ta98}.  

As $S$ and $S'$ are both smooth, the pairs $(\P^3, (1-\epsilon)S)$ and $(\P^3, (1-\epsilon)S')$ are both klt for any $0 < \epsilon <1$. In particular, one can use Noether-Fano inequality (\cite[Theorem 1.4]{Ta98}) to study the birational map $\Phi : \P^3 \dasharrow \P^3$ with $\Phi_*(S) = S'$. We will make a more specific choice of $\epsilon$ later.

Let $p : Y \to \P^3$ be a Hironaka resolution of indeterminacy of $\Phi$. We denote by $p' : Y \to \P^3$ the morphism such that $p' = \Phi \circ p$. We denote $S_Y = p_*^{-1}S$. Then $S_Y = (p')_*^{-1}S'$ as well. Denote by $\{E_j | j \in J\}$ the set of exceptional prime divisors of $p$. 

Let $\sL := \Phi_*^{-1}|H|$ be the proper transform of the complete linear system $|H| = |\sO_{\P^3}(1)|$ by $\Phi^{-1}$. Then $\sL$ is a linear subsystem of $|dH|$ for some $d \ge 1$ and $\sL$ has no fixed component. By abuse of notation, we also denote by $\sL$ a general element of the linear subsystem $\sL$. We define $a \in \Q_{>0}$ by
$$a = \frac{{\rm deg}\, S}{{\rm deg}\, \sL} = \frac{4}{d}\,\, .$$
We choose $0 < \epsilon <1$ so that
$$a\epsilon < 1\,\, , \,\, ad\epsilon = 4\epsilon <1\,\, .$$
As $K_{\P^3} = -4H$, $S = 4H$ and $\sL = dH = 4H/a$ in ${\rm NS}\,(\P^3)_{\R} = \R \cdot H$, we have
$$K_{\P^3} + (1-\epsilon)S + a\epsilon \sL = 0$$
in ${\rm NS}(\P^3)_{\R}$. We set 
$$D(t) := K_Y + (1 - \epsilon)S_Y + tp_*^{-1}\sL - p^*(K_{\P^3} + (1 - \epsilon)S + t\sL)$$
and define the rational number $c$ by
$$c := {\rm Max}\, \{0 < t \le 1 | D(t) \ge 0\}\,\, .$$
Here the inequality means that the $\R$-linear extension of the vanishing order ${\rm ord}_E(D(t))$ at all prime divisor $E \subset Y$, or equivalently at all $E_j$ ($j \in J$), is nonnegative. Note that $D(t) \ge 0$ if $t \le c$. We also define the slope $\mu \in \Q_{>0}$ by
$$\sL = \mu (-(K_{\P^3}+ (1-\epsilon)S))$$
in ${\rm NS}(\P^3)_{\R} = \R \cdot H$. In our case
$$\mu = \frac{1}{a\epsilon}\,\, .$$
Note that $K_{\P^3} + (1 -\epsilon)S + a\epsilon \sL = 0$ and this is nef. Thus, by the Noether-Fano inequality (\cite[Theorem 1.4]{Ta98}), $\Phi$ is an isomorphism if $c \ge a\epsilon$. As $\Phi$ is not an isomorphism by our assumption, it follows that 
$$c < a\epsilon\,\, .$$ 
As $0 < c < a\epsilon < 1$, there is then $E := E_j$ ($j \in J$) such that 
$${\rm ord}_{E}(D(a\epsilon)) < 0\,\, .$$
We choose and fix such an $E$ and set $F := p(E)_{{\rm red}}$. Then $F$ is either a closed point or an irreducible reduced curve on $\P^3$. Let 
$$\pi : Z \to \P^3$$
 be the blow up of $F$ and $E' \subset Z$ be the unique exceptional prime divisor that dominates $F$. Note that $\pi$ is the usual blow up over $\P^3 \setminus {\rm Sing}\, F$. Hence, $Z$ is smooth over $\P^3 \setminus {\rm Sing}\, F$. Here ${\rm Sing}\, F$ is the singular locus of $F$. In particular, $Z$ is smooth around the generic point $\eta_{E'}$ of $E'$. Note that ${\rm ord}_{E}(D(t))$ is determined at the generic point $\eta_{E}$ and does not depend on the birational model we choose. So, to compute ${\rm ord}_{E}(D(t))$, we may (and will) identify $(Y, E) = (Z, E')$. 

If $P$ is a point, then 
$${\rm ord}_{E}(D(a\epsilon)) \ge 2 - (1 - \epsilon) - a\epsilon \cdot {\rm mult}_P \sL 
\ge 1 + \epsilon - ad\epsilon > 0$$
by our choice of $\epsilon$. Here ${\rm mult}_P \sL$ is the multiplicity of (a general element of) $\sL$ at $P$. Similarly, if $F$ is a curve such that $F \not\subset S$, then 
$${\rm ord}_{E}(D(a\epsilon)) = 1 + 0 - a\epsilon \cdot {\rm mult}_F \sL 
\ge 1 - ad\epsilon > 0$$  
by our choice of $\epsilon$. Here ${\rm mult}_F \sL$ is the multiplicity of $\sL$ at a general point $P \in F$. 

Thus $F$ has to be a curve and $F \subset S$, as ${\rm ord}_{E}(D(a\epsilon)) < 0$. 

From now, we will work in $\pi : Z \to \P^3$ and will write $E'$ by $E$. 
Set 
$$m := {\rm mult}_F \sL\,\, .$$ 
Then 
$$0 > {\rm ord}_{E}(D(a\epsilon)) = 1 -(1-\epsilon) - a\epsilon m = \epsilon(1 - am)\,\, ,$$
and therefore 
$$am > 1\,\, .$$
Take a general plane $H \subset \P^3$. Then
$$H \cap F = \{p_1, \ldots , p_{{\rm deg}\, F}\} \subset F \setminus {\rm Sing}\, F\,\, .$$
In particular, $\pi|_{\pi^*H} : \pi^*H \to H$ is a usual blow-up at the $k$ points $p_i \in H$. Let $e_i$ be the exceptional curve over $p_i$. As $H \cap {\rm Sing}\, F = \emptyset$, it follows that, around $\pi^*H \subset Z$, the Weil divisor $\sL_Z := \pi_{*}^{-1}\sL$ is Cartier and is of the form 
$$\sL_Z = \pi^*\sL - mE\,\, .$$
Thus 
$$\sL_Z|_{\pi^*H} = \pi^*(\sL|_{H}) - \sum_{i=1}^{{\rm deg}\, F} me_i\,\, .$$
Combing this with the fact that the linear system $\sL$ is movable and $H$ is general, we have
$$0 \le ((\sL_Z|_{\pi^*H})^2)_{\pi^*H} = ({\rm deg}\, \sL)^2 - m^2{\rm deg}\, F\,\, .$$
As ${\rm deg}\, \sL = d = 4/a$ and $am >1$, it follows that
$${\rm deg}\, F \le (\frac{4}{am})^2 < 16\,\, .$$
Now it suffices to prove that $F$ and $h := H|S$ are linearly independent in ${\rm NS}\, (S)$. Assume to the contrary that they are linearly dependent. Then $F = bh$ for some $b \in \Q_{>0}$. As $(h^2)_S = 4$ and $(F^2)_S$ is even as $S$ is a smooth K3 surface, it follows that $b$ is a positive integer. As ${\rm Pic}\, (S) \simeq {\rm NS}\,(S)$ and they are torsion free, it follows that $F = ah$ in ${\rm Pic}\, (S)$. As $H^{1}(\P^3, \sO_{\P^3}(n)) = 0$ for all $n \in \Z$, it follows that the restriction map 
$$H^{0}(\P^3, \sO_{\P^3}(b)) \to H^0(S, \sO_{S}(bh))$$
is surjective. So, there is a surface $T \subset \P^3$ of degree $b$ such that $F = S \cap T$ as schemes. Let $R$ be the proper transform of $T \cap H$ under the morphism $\pi : Z \to \P^3$. Here we recall that $H \subset \P^3$ is a general plane. Then we have 
$$0 \le (\sL_{Z}.R)_{Z} = (\sL.T.H)_{\P^3} - m(F.T)_{\P^3} = (1-am)(\sL.T.H)_{\P^3}\,\, .$$
For the last equality we used $F = S \cap T$ and $S = a\sL$ in ${\rm NS}\, (\P^3)_{\R}$. However, then $am \le 1$ as $(\sL.T.H)_{\P^3} >0$, a contradiction to the previous inequality $am >1$. This completes the proof of Proposition (\ref{takahashi}).

\end{document}